\documentclass[twoside]{article}
\usepackage{amsbsy}

\usepackage{mathptmx}
\usepackage[11pt]{moresize}
\usepackage[margin=3.5cm]{geometry}
\usepackage{amssymb,amsmath,latexsym, theorem}

\title{A note on solutions of the matrix equation $AXB=C$}

\author{I.V. Jovovi\' c and B.J. Male\v sevi\' c}

\date{18.07.2013.}

\newcommand{\qed}{\hfill \Box\vspace*{4mm}}
\newtheorem{theorem}{Theorem}[section]
\newtheorem{lemma}[theorem]{Lemma}

\newtheorem{example}[theorem]{Example}

\newenvironment{proof}[1][Proof:]{\begin{trivlist}
\item[\hskip \labelsep {\bfseries #1}]}{\end{trivlist}}

\def\MRU#1{ \mathop{ #1^{\!\pmb{\urcorner}}} }
\def\MRD#1{ \mathop{ #1_{\!\pmb{\lrcorner}}} }
\def\MLU#1{ \mathop{ {}^{\pmb{\ulcorner}\!\!\!}#1 } }
\def\MLD#1{ \mathop{ {}_{\pmb{\llcorner}\!\!\!}#1 } }

\renewcommand{\arraystretch}{1.0} 

\begin{document}

\thispagestyle{empty}

\begin{flushleft}
{\footnotesize
{\sc SCIENTIFIC PUBLICATIONS OF THE STATE UNIVERSITY OF NOVI PAZAR \\}
\textsc{Ser. A: Appl. Math. Inform. and Mech.}}
\end{flushleft}

\vspace*{7mm}

\begin{center}
\textbf{{\Large A note on solutions of the matrix equation \textit{AXB$\,$=$\,$C}}}

\vspace*{3mm}

\textbf{{\large Ivana V. Jovovi\' c}}
$\!$\footnote{School of Electrical Engineering, University of Belgrade, Serbia,
e-mail: ivana@etf.rs}
$\,$\textbf{and}$\;$
\textbf{{\large Branko J. Male\v sevi\' c}}
$\!$\footnote{School of Electrical Engineering, University of Belgrade, Serbia,
e-mail: malesevic@etf.rs \\ \hspace*{5.0 mm}(corresponding author)}
\end{center}

\vspace*{5mm}

\noindent
\begin{minipage}[b]{145 mm}
\textbf{Abstract:}
This paper deals with necessary and sufficient condition for consistency of the matrix equation $AXB=C$. We will be concerned
with the minimal number of free parameters in Penrose's formula $X = A^{(1)} C B^{(1)} + Y - A^{(1)}A Y B B^{(1)}$ for obtaining
the general solution of the matrix equation and we will establish the relation between the minimal number of free parameters and
the ranks of the matrices $A$ and $B$. The solution is described in the terms of Rohde's general form of the $\{1\}$-inverse of
the matrices $A$ and $B$. We will also use Kronecker product to transform the matrix equation $AXB=C$ into the linear system
$(B^T \otimes A)\,vec\,X \,=\, vec\,C$.
\end{minipage}


\section{Introduction}

\medskip

In this paper we consider matrix equation

\vspace*{2.0 mm}

\noindent
\begin{equation}
\label{AXB=C}
AXB=C,
\end{equation}
where $X$ is an $n \times k$ matrix of unknowns,
      $A$ is an $m \times n$ matrix of rank $a$,
      $B$ is a  $k \times l$ matrix of rank $b$,
  and $C$ is an $m \times l$ matrix, all over $\mathbb{C}$.
The set of all $m \times n$ matrices over the complex field $\mathbb{C}$ will be denoted by $\mathbb{C}^{m \times n}$, $m, n \in \mathbb{N}$.
The set of all $m \times n$ matrices over the complex field $\mathbb{C}$ of rank $a$ will be denoted~by~$\mathbb{C}_{a}^{m \times n}$.
We will write $A_{i\rightarrow}$ ($A_{\downarrow j}$) for the $i^{th}$ row (the $j^{th}$ column) of the matrix $A \in \mathbb{C}^{m \times n}$
and $vec \, A$ will denote an ordered stock of columns of $A$, i.e.
$$
vec \, A
=
\left[
\begin{array}{c}
A_{\downarrow 1} \\
A_{\downarrow 2} \\
\vdots           \\
A_{\downarrow n}
\end {array}
\right].
$$
Using the Kronecker product of the matrices $B^T$ and $A$ we can transform the matrix equation (\ref{AXB=C}) into linear system

\vspace*{1.0 mm}

\noindent
\begin{equation}
\label{vec}
(B^T \otimes A)\, vec \, X = vec \, C.
\end{equation}
For the proof we refer the reader to A. Graham \cite{Graham81}. Necessary and sufficient condition for consistency
of the linear system $Ax = c$, as well as the minimal number of free parameters in Penrose's formula $x = A^{(1)}c + (I - A^{(1)}A)y$
has been considered in the paper B. Male\v sevi\' c, I. Jovovi\' c, M.~Makragi\' c and B.~Radi\v ci\' c \cite{MalesevicJovovicMakragicRadicic13}.
We will here briefly sketch this results in the case of the linear system (\ref{vec}).

\break

Any matrix $X$ satisfying the equality $AXA=A$ is called $\{1\}$-inverse of $A$ and is denoted by $A^{(1)}$.
For each matrix $A \in \mathbb{C}_{a}^{m \times n}$ there are regular matrices
$P \in \mathbb{C}^{n \times n}$ and $Q \in \mathbb{C}^{m \times m}$ such that
\begin{equation}
\label{QAP=Ea}
QAP
=
E_{A}
=
\left[
\begin{array}{c|c}
I_{a} & 0    \\\hline
0     & 0
\end {array}
\right],
\end{equation}
where $I_{a}$ is $a \times a$ identity matrix. It can be easily seen that every $\{1\}$-inverse of the matrix $A$ can be
represented in the Rohde's form
\begin{equation}
\label{A(1)}
A^{(1)} =
P\left[
\begin{array}{c|c}
I_{a} & U     \\\hline
V     & W
\end {array}
\right]Q \, ,
\end{equation}
where $U=[u_{ij}]$, $V=[v_{ij}]$ and $W=[w_{ij}]$ are arbitrary matrices of corresponding dimensions $a \times (m-a)$,
$(n-a) \times a$ and $(n-a) \times (m-a)$ with mutually independent entries, see C. Rohde \cite{Rohde64}
and V. Peri\' c \cite{Peric82}. We will explore the minimal numbers of free parameters
in Penrose's formula $$X = A^{(1)} C B^{(1)} + Y - A^{(1)}A Y B B^{(1)}$$
for obtaining the general solution of the matrix equation (\ref{AXB=C}).
Some similar considerations can be found in papers B. Male\v sevi\' c and B.~Radi\v ci\' c
\cite{MalesevicRadicic11}, \cite{MalesevicRadicic11a}, \cite{MalesevicRadicic12} and \cite{MalesevicRadicic11b}.

\section{Matrix equation \textbf{\textit{AXB$\,$=$\,$C}} and the Kronecker product of the matrices
$\mbox{\textbf{\textit{B}}}^T\!$ and $\mbox{\textbf{\textit{A}}}$}

\medskip

The Kronecker product of matrices $A=[a_{ij}]_{m \times n} \in \mathbb{C}^{m \times n}$
and $B=[b_{ij}]_{k \times l} \in \mathbb{C}^{k \times l}$, denoted by $A \otimes B$,
is defined as block matrix
$$
A \otimes B =
\left[
\begin{array}{cccc}
a_{11}B & a_{12}B & \ldots & a_{1n}B    \\
a_{21}B & a_{22}B & \ldots & a_{2n}B    \\
\vdots  & \vdots  & \ddots & \vdots     \\
a_{m1}B & a_{m2}B & \ldots & a_{mn}B
\end {array}
\right].
$$
The matrix $A \otimes B$ is $mk \times nl$ matrix with $mn$ blocks $a_{ij}B$ of order $k \times l$.
Here we will mention some properties and rules for the Kronecker product.
Let $A \in \mathbb{C}^{m \times n}$, $B \in \mathbb{C}^{k \times l}$,
$C \in \mathbb{C}^{n \times r}$ and $D \in \mathbb{C}^{l \times s}$.
Then the following propositions holds:
\begin{itemize}
  \item $A^T \otimes B^T = (A \otimes B)^T$;
  \item $rank(A \otimes B) = rank(A) \; rank(B)$;
  \item $(A \otimes B) \cdot (C \otimes D) = (A \cdot C) \otimes (B \cdot D)\;\;$(mixed product rule);
  \item if $A$ and $B$ are regular $n \times n$ and $k \times k$  matrices, then $(A \otimes B)^{-1}=A^{-1} \otimes B^{-1}$.
\end{itemize}
The proof of these facts can be found in A. Graham \cite{Graham81} and  A. Ben--Israel and T.N.E. Greville~\cite{Ben--IsraelGreville03}.

\noindent
Matrix $A^{(1)} \otimes B^{(1)}$ is $\{1\}$-inverse of $A \otimes B$.
Using mixed product rule we have
$$
(A \otimes B)(A^{(1)} \otimes B^{(1)})(A \otimes B)=(A \cdot A^{(1)} \cdot A) \otimes (B \cdot B^{(1)} \cdot B) = A \otimes B \, .
$$

\noindent
Let $R \in \mathbb{C}^{k \times k}$  and $S \in \mathbb{C}^{l \times l}$ be regular matrices
such that
\begin{equation}
\label{RBS=Eb}
RBS
=
E_{B}
=
\left[
\begin{array}{c|c}
I_{b} & 0    \\\hline
0     & 0
\end {array}
\right].
\end{equation}

\break

\noindent
An $\{1\}$-inverse of the matrix $B$ can be represented in the Rohde's form
\begin{equation}
\label{B(1)}
B^{(1)} =
S\left[
\begin{array}{c|c}
I_{b} & M     \\\hline
N     & K
\end {array}
\right]R \, ,
\end{equation}
where $M=[m_{ij}]$, $N=[n_{ij}]$ and $K=[k_{ij}]$
are arbitrary matrices of corresponding dimensions $b \times (k-b)$, $(l-b) \times b$ and $(l-b) \times (k-b)$
with mutually independent entries.

\medskip

From now on, we will look more closely at the linear system (\ref{vec}).

\smallskip
\noindent
Firstly, by mixed product rule we obtain
$$
(S^T \otimes Q) \cdot (B^T \otimes  A) \cdot (R^T \otimes P)
=
(S^T \cdot B^T \cdot R^T) \otimes (Q \cdot A \cdot P)
=
{E_{B^T}} \otimes E_A \, .
$$
Unfortunately, the matrix
$$
{E_{B^T}} \otimes E_A
=
\left[
\begin{array}{c|c}
I_{b} \otimes E_A & 0     \\\hline
0                 & 0
\end {array}
\right]
=
\left[
\begin{array}{c|c}
\begin{array}{cccc}
E_{A} &       &        &      \\
      & E_{A} &        & 0    \\
      &       & \ddots &      \\
      & 0     &        & E_A  \\
\end {array}       &    0   \\\hline
0                  &    0
\end {array}
\right]
=
\left[
\begin{array}{c|c}
\begin{array}{ccccccc}
I_{a} & 0 &       &   &        &       &     \\
0     & 0 &       &   &        & 0     &     \\
      &   & I_{a} & 0 &        &       &     \\
      &   & 0     & 0 &        &       &     \\
      &   &       &   & \ddots &       &     \\
      & 0 &       &   &        & I_{a} & 0   \\
      &   &       &   &        & 0     & 0   \\
\end {array}       &    0   \\\hline
0                  &    0
\end {array}
\right]
$$
is not of the needed form $E_{B^T \otimes A}$. Equality ${E_{B^T}} \otimes E_A = E_{B^T \otimes A}$ holds for $b=1$.
Swapping the rows and the columns corresponding to blocks $I_a$ and to zero diagonal blocks we get required matrix $E_{B^T \otimes A}$.
If matrices $D$ and $G$ are the elementary matrices obtained by swapping rows and columns corresponding to mentioned blocks
of the identity matrices, then $D \cdot ({E_{B^T}} \otimes E_A) \cdot G = E_{B^T \otimes A}$.
Thus, we have
\begin{equation}
\label{Eab}
(D \cdot(S^T \otimes Q)) \cdot (B^T \otimes  A) \cdot ((R^T \otimes P) \cdot G) = E_{B^T \otimes A}
\end{equation}
and so an $\{1\}$-inverse of the matrix $B^T \otimes A$ can be represented in the Rohde's form
\begin{equation}
\label{BTkronekerA(1)}
(B^T \otimes A)^{(1)} =
(R^T \otimes P) \cdot G \cdot \left[
\begin{array}{c|c}
I_{ab} & F     \\ \hline
H      & L
\end {array}
\right] \cdot D \cdot(S^T \otimes Q) \, ,
\end{equation}
where $F=[f_{ij}]$, $H=[h_{ij}]$ and $L=[l_{ij}]$ are arbitrary matrices of corresponding dimensions $ab \times (ml-ab)$, $(nk-ab) \times ab$
and $(nk-ab) \times (ml-ab)$ with mutually independent entries. If the matrices $A$ and $B$ are square matrices, then $D=G^T$.

For the simplicity of notation, we will write  $\overline{c}_{\,a}$ ($\,\underline{c}_{\,a}$\,) for the submatrix
corresponding to the first (the last) $a$ coordinates of the vector $c$.
Now, we can rephrase Lemma 2.1. and Theorem 2.2. from the paper
B. Male\v sevi\' c, I. Jovovi\' c, M.~Makragi\' c and B.~Radi\v ci\' c \cite{MalesevicJovovicMakragicRadicic13}
in the case of the linear system (\ref{vec}). Let $C\,'$ be given by $C\,'=QCS$. Then $vec \, C\,' = (S^T \otimes Q) \; vec \, C$.

\begin{lemma}
\label{vecC'}
The linear system {\rm (\ref{vec})} has a solution if and only if the last $ml-ab$ coordinates of the vector $c'' = D \; vec \, C\,'$
are zeros, where $D$ is elementary matrix such that {\rm (\ref{Eab})} holds.
\end{lemma}

\break

\begin{theorem}
\label{main theorem Kronecker}
The vector
\begin{equation}
\label{Penrose}
vec \, X = (B^T \otimes A)^{(1)}vec\, C + (I-(B^T \otimes A)^{(1)} \cdot (B^T \otimes A)) y,
\end{equation}
$\,y \in \mathbb{C}^{nk \times 1}\,$ is an arbitrary column, is the general solution of the system {\rm \,(\ref{vec})\,},
if and only if the $\{1\}$-inverse $(B^T \otimes A)^{(1)}$ of the system matrix $B^T \otimes A$ has the form {\rm(\ref{BTkronekerA(1)})}
for arbitrary matrices $F$ and $L$ and the rows of the matrix $H(\overline{c}\,''_{ab}-\overline{y}\,'_{ab}) + \underline{y}\,'_{nk-ab}$
are free parameters, where $D \cdot(S^T \otimes Q) vec\, C
\,=\, c'' \,=\,
\left[
\begin{array}{c}
\overline{c}\,''_{ab}      \\\hline
0
\end {array}
\right]$
and
$G^{-1} \cdot ((R^{-1})^T \otimes P^{-1})y
\,=
y\,'
\,=\,
\left[
\begin{array}{c}
\overline{y}\,'_{ab}      \\\hline
\underline{y}\,'_{nk-ab}
\end {array}
\right]$.
\end{theorem}

In the paper B. Male\v sevi\' c, I. Jovovi\' c, M.~Makragi\' c and B.~Radi\v ci\' c \cite{MalesevicJovovicMakragicRadicic13}
we have seen that general solution (\ref{Penrose}) can be presented in the form
\begin{equation}
vec \, X
=
(R^T \otimes P) \cdot G \cdot\!
\left[
\begin{array}{c}
\overline{c}\,''_{ab}                                           \\ \hline
H(\overline{c}\,''_{ab} - \overline{y}\,'_{ab}) + \underline{y}\,'_{nk-ab}
\end {array}
\right].
\end{equation}

We illustrate this formula in the next example.

\vspace*{-2.5 mm}

\begin{example}
\label{Example}
We consider the matrix equation
$$
AXB=C,
$$
where
$A =
\left[
\begin{array}{rr}
1 & 2 \\
0 & 1 \\
1 & 1
\end{array}
\right]$,
$X =
\left[
\begin{array}{rrr}
x_{11} & x_{12} & x_{13}   \\
x_{21} & x_{22} & x_{23}
\end{array}
\right]$,
$B =
\left[
\begin{array}{rrr}
1 & 0 & 0 \\
0 & 1 & 1 \\
1 & 1 & 1
\end{array}
\right]$
and
$C =
\left[
\begin{array}{rrr}
-3 & -6 & -6 \\
-1 & -2 & -2 \\
-2 & -4 & -4
\end{array}
\right]$.

\medskip
Using the Kronecker product the matrix equation may be considered in the form of the equivalent linear system
$$
(B^T \otimes A) \cdot vec \, X = vec \, C,
$$
i.e.
$$
\left[
\begin{array}{rrrrrr}
1 & 2 & 0 & 0 & 1 & 2 \\
0 & 1 & 0 & 0 & 0 & 1 \\
1 & 1 & 0 & 0 & 1 & 1 \\
0 & 0 & 1 & 2 & 1 & 2 \\
0 & 0 & 0 & 1 & 0 & 1 \\
0 & 0 & 1 & 1 & 1 & 1 \\
0 & 0 & 1 & 2 & 1 & 2 \\
0 & 0 & 0 & 1 & 0 & 1 \\
0 & 0 & 1 & 1 & 1 & 1
\end{array}
\right]
\cdot
\left[
\begin{array}{r}
x_{11} \\
x_{21} \\
x_{12} \\
x_{22} \\
x_{13} \\
x_{23}
\end{array}
\right]  =
\left[
\begin{array}{r}
-3 \\
-1 \\
-2 \\
-6 \\
-2 \\
-4 \\
-6 \\
-2 \\
-4
\end{array}
\right].
$$
Matrices
$Q
=\!
\left[
\begin{array}{rrr}
\! 1 & 0 & 0    \\
\! 0 & 1 & 0    \\
\!-1 & 1 & 1
\end{array}
\right]$
and
$P
=\!
\left[
\begin{array}{rr}
1 & \!-2    \\
0 & \! 1
\end{array}
\right]$
are regular matrices such that
$QAP
=
E_A
=
\left[
\begin{array}{rr}
1 & 0   \\
0 & 1   \\
0 & 0
\end{array}
\right]$
and matrices
$R
=
\left[
\begin{array}{rrr}
\! 1 & \! 0 & 0    \\
\! 0 & \! 1 & 0   \\
\!-1 & \!-1 & 1
\end{array}
\right]$
and
$S =
\left[
\begin{array}{rrr}
1 & 0 & \! 0 \\
0 & 1 & \!-1 \\
0 & 0 & \! 1
\end{array}
\right]$
are regular matrices such that
$
RBS
=
$
$
E_B
=
\left[
\begin{array}{rrr}
1 & 0 & 0 \\
0 & 1 & 0  \\
0 & 0 & 0
\end{array}
\right]$.
Therefore,
$$
{E_{B^T}} \otimes E_A =
\left[
\begin{array}{rrrrrr}
1 & 0 & 0 & 0 & 0 & 0 \\
0 & 1 & 0 & 0 & 0 & 0 \\
0 & 0 & 0 & 0 & 0 & 0 \\
0 & 0 & 1 & 0 & 0 & 0 \\
0 & 0 & 0 & 1 & 0 & 0 \\
0 & 0 & 0 & 0 & 0 & 0 \\
0 & 0 & 0 & 0 & 0 & 0 \\
0 & 0 & 0 & 0 & 0 & 0 \\
0 & 0 & 0 & 0 & 0 & 0
\end{array}
\right]
\mbox{\;and for matrix\;\;}
D =
\left[
\begin{array}{rrrrrrrrr}
1 & 0 & 0 & 0 & 0 & 0 & 0 & 0 & 0 \\
0 & 1 & 0 & 0 & 0 & 0 & 0 & 0 & 0 \\
0 & 0 & 0 & 1 & 0 & 0 & 0 & 0 & 0 \\
0 & 0 & 0 & 0 & 1 & 0 & 0 & 0 & 0 \\
0 & 0 & 1 & 0 & 0 & 0 & 0 & 0 & 0 \\
0 & 0 & 0 & 0 & 0 & 1 & 0 & 0 & 0 \\
0 & 0 & 0 & 0 & 0 & 0 & 1 & 0 & 0 \\
0 & 0 & 0 & 0 & 0 & 0 & 0 & 1 & 0 \\
0 & 0 & 0 & 0 & 0 & 0 & 0 & 0 & 1
\end{array}
\right]
$$
we have
$$
E_{B^T \otimes A}
=
D \cdot ({E_{B^T}} \otimes E_A)
=
\left[
\begin{array}{rrrrrr}
1 & 0 & 0 & 0 & 0 & 0 \\
0 & 1 & 0 & 0 & 0 & 0 \\
0 & 0 & 1 & 0 & 0 & 0 \\
0 & 0 & 0 & 1 & 0 & 0 \\
0 & 0 & 0 & 0 & 0 & 0 \\
0 & 0 & 0 & 0 & 0 & 0 \\
0 & 0 & 0 & 0 & 0 & 0 \\
0 & 0 & 0 & 0 & 0 & 0 \\
0 & 0 & 0 & 0 & 0 & 0
\end{array}
\right].
$$
Let us remark that $E_{B^T \otimes A} = (D \cdot (S^T \otimes Q)) \cdot (B^T \otimes  A) \cdot (R^T \otimes P)$, and hence
according to the Theorem~\ref{main theorem Kronecker}
$vec \, X
=
(R^T \otimes P) \cdot
\left[
\begin{array}{c}
\overline{c}\,''_{4}                                       \\ \hline
H(\overline{c}\,''_{4} - \overline{y}\,'_{4}) + \underline{y}\,'_{2}
\end {array}
\right]$
is the general solution of the linear system iff elements of the column
$H(\overline{c}\,''_{4}-\overline{y}\,'_{4}) + \underline{y}\,'_{2}$
are two mutually independent parameters $\alpha_1$ i $\alpha_2$ for the~vector
$$
\begin{array}{rcl}
c\,''
\!\!&\!\!=\!\!&\!\!
\left[
\begin{array}{c}
\overline{c}\,''_{4}                \\ \hline
0
\end{array}
\right] =
(D \cdot (S^T \otimes Q)) \cdot vec \, C
=                                  \\[2.5ex]
\!\!&\!\!=\!\!&\!\!
\left[
\begin{array}{rrrrrrrrr}
 1 & 0 & 0 & 0 & 0 & 0 & 0 & 0 & 0 \\
 0 & 1 & 0 & 0 & 0 & 0 & 0 & 0 & 0 \\
 0 & 0 & 0 & 1 & 0 & 0 & 0 & 0 & 0 \\
 0 & 0 & 0 & 0 & 1 & 0 & 0 & 0 & 0 \\
-1 & 1 & 1 & 0 & 0 & 0 & 0 & 0 & 0 \\
 0 & 0 & 0 &-1 & 1 & 1 & 0 & 0 & 0 \\
 0 & 0 & 0 &-1 & 0 & 0 & 1 & 0 & 0 \\
 0 & 0 & 0 & 0 &-1 & 0 & 0 & 1 & 0 \\
 0 & 0 & 0 & 1 &-1 &-1 &-1 & 1 & 1
\end{array}
\right]
\cdot
\left[
\begin{array}{r}
-3 \\
-1 \\
-2 \\
-6 \\
-2 \\
-4 \\
-6 \\
-2 \\
-4
\end{array}
\right]
\;=\;
\left[
\begin{array}{r}
-3 \\
-1 \\
-6 \\
-2 \\\hline
 0 \\
 0 \\
 0 \\
 0 \\
 0
\end{array}
\right].
\end{array}
$$
Finally, the general solution of the linear system is
$$
vec \, X
=
(R^T \otimes P) \cdot\!
\left[
\begin{array}{r}
-3 \\
-1 \\
-6 \\
-2 \\ \hline
\alpha_{1} \\
\alpha_{2}
\end{array}
\right]
 =
\left[
\begin{array}{rrrrrr}
1 &-2 & 0 & 0 &-1 & 2 \\
0 & 1 & 0 & 0 & 0 &-1 \\
0 & 0 & 1 &-2 &-1 & 2 \\
0 & 0 & 0 & 1 & 0 &-1 \\
0 & 0 & 0 & 0 & 1 &-2 \\
0 & 0 & 0 & 0 & 0 & 1
\end{array}
\right]
\cdot
\left[
\begin{array}{r}
-3 \\
-1 \\
-6 \\
-2 \\ \hline
\alpha_{1} \\
\alpha_{2}
\end{array}
\right]
 =
\left[
\begin{array}{r}
-1         -  \alpha_{1} + 2\alpha_{2}  \\
-1         -  \alpha_{2}                \\
-2         -  \alpha_{1} + 2\alpha_{2}  \\
-2         -  \alpha_{2}                \\
\alpha_{1} - 2\alpha_{2}                \\
\alpha_{2}
\end{array}
\right].
$$
\end{example}

\break

\section{Matrix equation \textbf{\textit{AXB$\,$=$\,$C}} and the \textbf{\{1\}}--$\,$inverses
of the matrices \textbf{\textit{A}}~and~\textbf{\textit{B}}}

\medskip

In this section we indicate how technique of an $\{1\}$-inverse may be used to obtain the necessary
and sufficient condition for an existence of a general solution of the matrix equation (\ref{AXB=C})
without using Kronecker product.
We will use the symbols $\MLU{C}_{a,b}$, $\MLD{C}_{a,b}$, $\MRU{C}_{a,b}$ and $\MRD{C}_{a,b}$
for the submatrices of the matrix $C$ corresponding to the first $a$ rows and $b$ columns,
the last $a$ rows and the first $b$ columns, the first $a$ rows and the last $b$ columns,
the last $a$ rows and $b$ columns, respectively.
\begin{lemma}
\label{C'=QCS}
The matrix equation {\rm (\ref{AXB=C})} has a solution if and only if
the last $m-a$ rows and $l-b$ columns of the matrix $C\,'=QCS$ are zeros, where
$Q \in \mathbb{C}^{m \times m}$ and $S \in \mathbb{C}^{l \times l}$
are regular matrices such that {\rm(\ref{QAP=Ea})} and {\rm(\ref{RBS=Eb})} hold.
\end{lemma}
\begin{proof}
The matrix equation (\ref{AXB=C}) has a solution if and only if $C=AA^{(1)}CB^{(1)}B$, see R.~Penrose \cite{Penrose55}.
Since $A^{(1)}$ and $B^{(1)}$ are described by the equations (\ref{A(1)}) and (\ref{B(1)}), it follows that
$$
AA^{(1)}
=
AP\left[
\begin{array}{c|c}
I_{a} & U     \\ \hline
V     & W
\end{array}
\right]Q
 =
Q^{-1}\left[
\begin{array}{c|c}
I_{a} & U     \\ \hline
0     & 0
\end{array}
\right]Q
$$
and
$$
B^{(1)}B =
S\left[
\begin{array}{c|c}
I_{b} & M     \\\hline
N     & K
\end{array}
\right]RB
 =
S\left[
\begin{array}{c|c}
I_{b} & 0     \\\hline
N     & 0
\end{array}
\right]S^{-1}.
$$
Hence, since $Q$ and $S$ are regular matrices we have the following equivalences
$$
\!\!\!\!\!\!\!\!\!
\begin{array}{lcl}
C
\!=\!
AA^{(1)}CB^{(1)}B
\!\!&\!\!\Longleftrightarrow\!\!&\!\!
QCS
\!=\!
QAA^{(1)}CB^{(1)}BS
\;\;\;\;\mathop{\Longleftrightarrow}\limits^{^{{\mbox {\tiny$C\,'=QCS$}}}}\;\;\;\;
C\,'
\!=\!
\left[
\begin{array}{c|c}
I_{a} & U     \\ \hline
0     & 0
\end{array}
\right]
C\,'
\left[
\begin{array}{c|c}
I_{b} & 0     \\ \hline
N     & 0
\end{array}
\right]                                                         \\[3ex]
\!\!\!\!&\!\!\!\!\Longleftrightarrow\!\!\!\!&\!\!\!\!
\left[
\begin{array}{c|c}
\MLU{C\,'}_{a, b}      & \MRU{C\,'}_{a, l-b}     \\ \hline
\MLD{C\,'}_{m-a, b}    & \MRD{C\,'}_{m-a, l-b}
\end{array}
\right]=
\left[
\begin{array}{c|c}
I_{a} & U     \\ \hline
0     & 0
\end{array}
\right]
\left[
\begin{array}{c|c}
\MLU{C\,'}_{a, b}       & \MRU{C\,'}_{a, l-b}     \\ \hline
\MLD{C\,'}_{m-a, b}     & \MRD{C\,'}_{m-a, l-b}
\end{array}
\right]
\left[
\begin{array}{c|c}
I_{b} & 0     \\ \hline
N     & 0
\end{array}
\right]                                                     \\[3ex]
\!\!\!\!&\!\!\!\!\Longleftrightarrow\!\!\!\!&\!\!\!\!
\left[
\begin{array}{c|c}
\MLU{C\,'}_{a, b}       & \MRU{C\,'}_{a, l-b}     \\ \hline
\MLD{C\,'}_{m-a, b}     & \MRD{C\,'}_{m-a, l-b}
\end{array}
\right]
\!=\!
\left[
\begin{array}{c|c}
\MLU{C\,'}_{\!a, b} + U \MLD{C\,'}_{\!m-a, b} + \MRU{C\,'}_{\!a, l-b} N + U \MRD{C\,'}_{n-a, l-b} N  & 0     \\ \hline
0                                                                                                      & 0
\end{array}
\right],
\end{array}
$$
for
$
\,C\,'
=
\left[
\begin{array}{c|c}
\MLU{C\,'}_{a, b}       & \MRU{C\,'}_{a, l-b}     \\ \hline
\MLD{C\,'}_{m-a, b}     & \MRD{C\,'}_{m-a, l-b}
\end{array}
\right]
$.
Therefore, we conclude
$$
C = AA^{(1)}CB^{(1)}B
\;\;\Longleftrightarrow\;\;
\mbox{$\MRU{C\,'}_{a, l-b}$}=0
\;\wedge\;
\mbox{$\MLD{C\,'}_{m-a, b}$}=0
\;\wedge\;
\mbox{$\MRD{C\,'}_{\, m-a, l-b}$}=0\,.
\quad\qed
$$
\end{proof}

\vspace*{-5.0 mm}

As we have seen in the Lemma \ref{vecC'} the matrix equation (\ref{AXB=C}) has a solution if and only if the last $ml-ab$ coordinates
of the column $c'' = D\,vec\,C\,'$ are zeros, where $D$ is elementary matrix such that {\rm (\ref{Eab})} holds.
Here we obtain the same result without using Kronecker product. The last $m(l-b)$ elements of the column $vec \,C\,'$
are zeros and there are $b$ blocks of $m-a$ zeros. Multiplying by the left column $vec \, C\,'$ with elementary matrix $D$
switches the rows corresponding to this zeros blocks under the blocks $\MLU{C\,'}_{a, b \; \downarrow \,i}$, $1 \leq i \leq b$.
Hence, the last $m(l-b) + (m-a)b = ml-ab$ entries of the column $c''$ are zeros.

\smallskip

Furthermore, we give a new form of the general solution of the matrix equation (\ref{AXB=C}) using $\{1\}$-inverses of the matrices $A$ and $B$.

\break

\noindent
\begin{theorem}
\label{main theorem}
The matrix
\begin{equation}
X
=
A^{(1)}CB^{(1)}+Y-A^{(1)}AYBB^{(1)},
\end{equation}
$Y \in \mathbb{C}^{n \times k}$ is an arbitrary matrix, is the general solution of the matrix equation {\rm (\ref{AXB=C})}
if and only if the $\{1\}$-inverses $A^{(1)}$ and $B^{(1)}$ of the  matrices $A$ and $B$ have the forms {\rm(\ref{A(1)})}
and {\rm(\ref{B(1)})} for arbitrary matrices $U$, $W$, $N$ and $K$ and the entries of the matrices
\begin{equation}
\mbox{$
V{\big (}\MLU{C\,'}_{a, b}-\MLU{Y\,'}_{a, b}{\big )}   + \MLD{Y\,'}_{n-a, b}\, , \; \;
 {\big (}\MLU{C\,'}_{a, b}-\MLU{Y\,'}_{a, b}{\big )}M  + \MRU{Y\,'}_{a, k-b}\, , \; \;
V{\big (}\MLU{C\,'}_{a, b}-\MLU{Y\,'}_{a, b}{\big )}M  + \MRD{Y\,'}_{\,n-a, k-b}
$}
\end{equation}
are mutually independent free parameters, where
\begin{equation}
\mbox{$
QCS
=
C\,'
=
\left[
\begin{array}{c|c}
\MLU{C\,'}_{a, b}  & 0     \\ \hline
0                  & 0
\end{array}
\right]$}
\;\;\;\mbox{\it and}\;\;\;
\mbox{$
P^{-1}YR^{-1} \,=\; Y \;=\;
\left[
\begin{array}{c|c}
\MLU{Y\,'}_{a, b}    & \MRU{Y\,'}_{a, k-b}     \\ \hline
\MLD{Y\,'}_{n-a, b}  & \MRD{Y\,'}_{\,n-a, k-b}
\end{array}
\right]$}.
\end{equation}
\end{theorem}
\begin{proof}
Since the $\{1\}$-inverses $A^{(1)}$ and $B^{(1)}$ of the matrices $A$ and $B$ have the forms (\ref{A(1)}) and (\ref{B(1)}),
the solution of the system $X=A^{(1)}CB^{(1)}+Y-A^{(1)}AYBB^{(1)}$ can be represented in the form
$$
\begin{array}{rcl}
X
\!\!&\!\!=\!\!&\!\!
P
\left[
\begin{array}{c|c}
I_{a} & U     \\ \hline
V     & W
\end{array}
\right]
Q C S \left[
\begin{array}{c|c}
I_{b} & M     \\ \hline
N     & K
\end{array}
\right]
R\;+\;Y\;-\;P
\left[
\begin{array}{c|c}
I_{a} & U     \\ \hline
V     & W
\end{array}
\right]
Q A P P^{-1} Y R^{-1} R B S
\left[
\begin{array}{c|c}
I_{b} & M     \\ \hline
N     & K
\end{array}
\right]
R.
\end{array}$$
According to Lemma \ref{C'=QCS} and from (\ref{QAP=Ea}) and (\ref{RBS=Eb}) we have
$$
\begin{array}{rcl}
X
\!\!&\!\!=\!\!&\!\!
P\left[
\begin{array}{c|c}
I_{a} & U     \\ \hline
V     & W
\end {array}
\right]
\left[
\begin{array}{c|c}
\MLU{C\,'}_{a, b} & 0  \\ \hline
0                 & 0
\end{array}
\right]
\left[
\begin{array}{c|c}
I_{b} & M     \\ \hline
N     & K
\end{array}
\right]
R \;\;+\;\; Y                                    \\[2.0 ex]
\!\!&\!\!-\!\!&\!\!
P
\left[
\begin{array}{c|c}
I_{a} & U     \\ \hline
V     & W
\end{array}
\right]
\left[
\begin{array}{c|c}
I_{a} & 0     \\ \hline
0     & 0
\end{array}
\right]
P^{-1}YR^{-1}
\left[
\begin{array}{c|c}
I_{b} & 0     \\ \hline
0     & 0
\end{array}
\right]
\left[
\begin{array}{c|c}
I_{b} & M     \\ \hline
N     & K
\end{array}
\right]
R.
\end{array}
$$
Furthermore, we obtain
\renewcommand{\arraystretch}{1.16} 
$$
\begin{array}{rcl}
X
\!\!&\!\!=\!\!&\!\!
P\left[
\begin{array}{c|c}
 \MLU{C\,'}_{a, b} & \MLU{C\,'}_{a, b}M    \\ \hline
V\MLU{C\,'}_{a, b} &  V\MLU{C\,'}_{a, b}M
\end{array}
\right]R
\;+\;
Y
\;-\;
P\left[
\begin{array}{c|c}
I_{a} & 0     \\ \hline
V     & 0
\end{array}
\right]
P^{-1}YR^{-1}
\left[
\begin{array}{c|c}
I_{b} & M     \\ \hline
0     & 0
\end{array}
\right]R                                                        \\[3.0 ex]
\!\!&\!\!=\!\!&\!\!
P\left(
\left[
\begin{array}{c|c}
 \MLU{C\,'}_{a, b} &    \MLU{C\,'}_{a, b}M    \\ \hline
V\MLU{C\,'}_{a, b} &  V \MLU{C\,'}_{a, b}M
\end{array}
\right]
\;+\;
\,Y\,'
\;-\;
\left[
\begin{array}{c|c}
I_{a} & 0     \\ \hline
V     & 0
\end{array}
\right]
Y\,'
\left[
\begin{array}{c|c}
I_{b} & M     \\ \hline
0     & 0
\end{array}
\right]
\right)R                                                        \\[3.0 ex]
\!\!&\!\!=\!\!&\!\!
P\left(
\left[
\begin{array}{c|c}
  \MLU{C\,'}_{a, b} &  \MLU{C\,'}_{a, b}M   \\ \hline
V \MLU{C\,'}_{a, b} & V\MLU{C\,'}_{a, b}M
\end{array}
\right]
\;+\;
\right.                                               \\[3.0 ex]
\!\!&\!\!+\!\!&\!\!
\left. \;\;\;\;\;\;\left[
\begin{array}{c|c}
\MLU{Y\,'}_{a, b}   & \MRU{Y\,'}_{\,a,k-b}    \\ \hline
\MLD{Y\,'}_{n-a, b} & \MRD{Y\,'}_{\,n-a, k-b}
\end{array}
\right]
\;-\;
\left[
\begin{array}{c|c}
I_{a} & 0     \\ \hline
V     & 0
\end{array}
\right]
\left[
\begin{array}{c|c}
\MLU{Y\,'}_{a, b}   & \MRU{Y\,'}_{\,a,k-b}    \\ \hline
\MLD{Y\,'}_{n-a, b} & \MRD{Y\,'}_{\,n-a, k-b}
\end{array}
\right]
\left[
\begin{array}{c|c}
I_{b} & M     \\ \hline
0     & 0
\end{array}
\right]
\right)R                \\[2.5 ex]
\!\!&\!\!=\!\!&\!\!
P\left(
\left[
\begin{array}{c|c}
 \MLU{C\,'}_{a,b} &  \MLU{C\,'}_{a,b}M        \\ \hline
V\MLU{C\,'}_{a,b} & V\MLU{C\,'}_{a,b}M
\end{array}
\right]
\;+\;
\left[
\begin{array}{c|c}
\MLU{Y\,'}_{a, b}   & \MRU{Y\,'}_{\,a,k-b}    \\ \hline
\MLD{Y\,'}_{n-a, b} & \MRD{Y\,'}_{\,n-a, k-b}
\end{array}
\right]
\;-\;
\left[
\begin{array}{c|c}
 \MLU{Y\,'}_{a, b}  &   \MLU{Y\,'}_{a, b}M     \\ \hline
V\MLU{Y\,'}_{a, b}  &  V\MLU{Y\,'}_{a, b}M
\end{array}
\right]
\right)R \, ,
\end{array}
$$
where $Y'=P^{-1}YR^{-1}$.
We now conclude
\begin{equation}
\label{solution}
X
\;=\;
P
\left[
\begin{array}{c|c}
  \MLU{C\,'}_{a,b}                          &  (\MLU{C\,'}_{a,b} - \MLU{Y\,'}_{a, b})M + \MRU{Y\,'}_{\,a,k-b}           \\ \hline
V(\MLU{C\,'}_{a,b}-\MLU{Y\,'}_{a, b}) + \MLD{Y\,'}_{n-a, b} &  V(\MLU{C\,'}_{a,b} - \MLU{Y\,'}_{a, b})M + \MRD{Y\,'}_{\,n-a, k-b}
\end{array}
\right] R.
\end{equation} \renewcommand{\arraystretch}{1.0} 
$\!\!\!$According to the Theorem \ref{main theorem Kronecker} the general solution of the equation (\ref{AXB=C})
has $nk-ab$ free parameters. Therefore, since the matrices $P$ and $R$ are regular we deduce that
the solution (\ref{solution}) is the general if and only if the entries of the matrices separately
$$
\mbox{$
V{\big (}\MLU{C\,'}_{a, b}-\MLU{Y\,'}_{a, b}{\big )}   + \MLD{Y\,'}_{n-a, b}\, , \; \;
 {\big (}\MLU{C\,'}_{a, b}-\MLU{Y\,'}_{a, b}{\big )}M  + \MRU{Y\,'}_{a, k-b}\, , \; \;
V{\big (}\MLU{C\,'}_{a, b}-\MLU{Y\,'}_{a, b}{\big )}M  + \MRD{Y\,'}_{\,n-a, k-b}
$}
$$
are $nk-ab$ free parameters. $\qed$
\end{proof}
We can illustrate the Theorem \ref{main theorem} on the following two examples.
\begin{example}
\label{Example2}
Consider again the matrix equation from previous example
$$
AXB=C,
$$
where
$A =
\left[
\begin{array}{rr}
1 & 2 \\
0 & 1 \\
1 & 1
\end{array}
\right]$,
$X =
\left[
\begin{array}{rrr}
x_{11} & x_{12} & x_{13}   \\
x_{21} & x_{22} & x_{23}
\end{array}
\right]$,
$B =
\left[
\begin{array}{rrr}
1 & 0 & 0 \\
0 & 1 & 1 \\
1 & 1 & 1
\end{array}
\right]$
and
$C =
\left[
\begin{array}{rrr}
-3 & -6 & -6 \\
-1 & -2 & -2 \\
-2 & -4 & -4
\end{array}
\right]$.

\medskip
We have
$$
C\,'
=
QCS
=
\left[
\begin{array}{rrr}
 1 & 0 & 0    \\
 0 & 1 & 0    \\
-1 & 1 & 1
\end{array}
\right]
\left[
\begin{array}{rrr}
-3 & -6 & -6 \\
-1 & -2 & -2 \\
-2 & -4 & -4
\end{array}
\right]
\left[
\begin{array}{rrr}
1 & 0 & 0 \\
0 & 1 &-1 \\
0 & 0 & 1
\end{array}
\right]=
\left[
\begin{array}{rrr}
-3 & -6 & 0 \\
-1 & -2 & 0 \\
 0 &  0 & 0
\end{array}
\right]
$$
and
$$
\mbox{$\MLU{C\,'}_{2,2}$}
=
\left[
\begin{array}{rr}
-3 & -6  \\
-1 & -2  \\
\end{array}
\right].
$$
The solution of the matrix equation $AXB=C$ is
$$
\begin{array}{lcr}
X
\;=\;
P\left[
\begin{array}{rr|r}
-3 & -6 & \alpha_1 \\
-1 & -2 & \alpha_2
\end{array}
\right]R
\!\!&\!\!=\!\!&\!\!
\left[
\begin{array}{rr}
1 & -2    \\
0 &  1
\end{array}
\right]
\left[
\begin{array}{rr|r}
-3 & -6 & \alpha_1 \\
-1 & -2 & \alpha_2
\end{array}
\right]
\left[
\begin{array}{rrr}
 1 &  0 & 0   \\
 0 &  1 & 0   \\
-1 & -1 & 1
\end{array}
\right]                                                               \\[4ex]
\!\!&\!\!=\!\!&\!\!
\left[
\begin{array}{rrr}
-1-\alpha_1+2\alpha_2 & -2-\alpha_1+2\alpha_2 & \alpha_1-2\alpha_2 \\
-1-\alpha_2           & -2-\alpha_2           & \alpha_2
\end{array}
\right].
\end{array}
$$
\end{example}

\begin{example}
\label{Example3}
We now consider the matrix equation
$$
AXB=C,
$$
where
$A = \left[
\begin{array}{rrr}
1 & 3 & 2 \\
2 & 6 & 4 \\
1 & 3 & 2
\end{array}
\right]$,
$X = \left[
\begin{array}{rr}
x_{11} & x_{12} \\
x_{21} & x_{22} \\
x_{31} & x_{32} \\
\end{array}
\right]$,
$B  = \left[
\begin{array}{rr}
 1 & -3   \\
-2 &  6
\end{array}
\right]$
and
$C=\left[
\begin{array}{rrr}
2  & -6  \\
4  & -12 \\
2  & -6
\end{array}
\right]$.

\medskip
For regular matrices
$Q = \left[
\begin{array}{rrr}
 1 & 0 & 0 \\
-2 & 1 & 0 \\
-1 & 0 & 1
\end{array}
\right]$
and
$P = \left[
\begin{array}{rrr}
1 &-3 &-2 \\
0 & 1 & 0 \\
0 & 0 & 1
\end{array}
\right]$
the following equality
$QAP
=
E_A
=
\!\left[
\begin{array}{ccc}
1 & 0 & 0 \\
0 & 0 & 0 \\
0 & 0 & 0
\end{array}
\right]$
holds. Thus, rank of the matrix $A$ is $a=1$. There are regular matrices
\mbox{$R =\!
\left[
\begin{array}{rr}
\!\! 1 & 0  \\
\!\! 2 & 1
\end{array}
\right]$}
and
$S =\!
\left[
\begin{array}{rrr}
1 & 3 \\
0 & 1
\end{array}
\right]$
such that $RBS=E_B=\left[
\begin{array}{rrr}
1 & 0 \\
0 & 0
\end{array}
\right]$
holds. Thus, rank of the matrix $B$ is $b=1$. Since the ranks of the matrices $A$ and $B$ are $a=b=1$, according to the Lemma \ref{C'=QCS}
all entries of the last column and the last two rows of the matrix $C\,'=QCS$ are zeros, i.e. we get that the matrix $C\,'$ is of the form
$$
C\,'
=
QCS
=
\left[
\begin{array}{rrr}
 1 & 0 & 0 \\
-2 & 1 & 0 \\
-1 & 0 & 1
\end{array}
\right]
\left[
\begin{array}{rrr}
2  & -6  \\
4  & -12 \\
2  & -6
\end{array}
\right]
\left[
\begin{array}{rrr}
1 & 3 \\
0 & 1
\end{array}
\right]=
\left[
\begin{array}{rrr}
2  &  0 \\
0  &  0 \\
0  &  0
\end{array}
\right].
$$
Applying the Theorem \ref{main theorem}, we obtain the general solution of the given matrix equation
$$
\begin{array}{lcl}
X
\;=\;
P\left[
\begin{array}{r|rr}
2         &  \alpha   \\ \hline
\beta_1   &  \gamma_1 \\
\beta_2   &  \gamma_2
\end{array}
\right]R
\!\!&\!\!=\!\!&\!\!
\left[
\begin{array}{rrr}
1 &-3 &-2 \\
0 & 1 & 0 \\
0 & 0 & 1
\end{array}
\right]
\left[
\begin{array}{r|rr}
2         &  \alpha   \\ \hline
\beta_1   &  \gamma_1 \\
\beta_2   &  \gamma_2
\end{array}
\right]
\left[
\begin{array}{rr}
 1 & 0  \\
 2 & 1
\end{array}
\right]                                                               \\[4ex]
\!\!&\!\!=\!\!&\!\!
\left[
\begin{array}{rr}
2+2\alpha-3\beta_1-2\beta_2-6\gamma_1-4\gamma_2   & \alpha-3\gamma_1-2\gamma_2 \\
\beta_1+2\gamma_1                                 & \gamma_1                 \\
\beta_2+2\gamma_2                                 & \gamma_2
\end{array}
\right].
\end{array}
$$
\end{example}

\bigskip

\noindent
{\bf ACKNOWLEDGMENT.}
Research is partially supported by the Ministry of Science and Education of the Republic of Serbia, Grant No.174032.

\end{document}